\documentclass[12pt]{amsart}
\usepackage{amsmath,amsthm,amsfonts,amssymb,mathrsfs}
\date{\today}

\usepackage{color}
\input xymatrix
\xyoption{all}

\usepackage{hyperref}

\newtheorem{theorem}{Theorem}

\newtheorem{proposition}[theorem]{Proposition}
\newtheorem{corollary}[theorem]{Corollary}
\newtheorem{lemma}[theorem]{Lemma}
\theoremstyle{definition}
\newtheorem{example}[theorem]{Example}
\newtheorem{remark}[theorem]{Remark}

\begin{document}

\title[On feebly compact semitopological symmetric inverse semigroups of ...]{On feebly compact semitopological symmetric inverse semigroups of a bounded finite rank}

\author{Oleg~Gutik}
\address{Faculty of Mathematics, National University of Lviv,
Universytetska 1, Lviv, 79000, Ukraine}
\email{o\underline{\hskip5pt}\,gutik@lnu.edu.ua,
ovgutik@yahoo.com}



\begin{abstract}
We study feebly compact shift-continuous $T_1$-topologies on the symmetric inverse semigroup $\mathscr{I}_\lambda^n$ of finite transformations of the rank $\leqslant n$. For any positive integer $n\geqslant2$ and any infinite cardinal $\lambda$ a Hausdorff countably pracompact non-compact shift-continuous topology on $\mathscr{I}_\lambda^n$ is constructed. We show that for an arbitrary positive integer $n$ and an arbitrary infinite cardinal $\lambda$ for a $T_1$-topology $\tau$ on $\mathscr{I}_\lambda^n$ the following conditions are equivalent: $(i)$ $\tau$ is countably pracompact; $(ii)$ $\tau$ is feebly compact; $(iii)$ $\tau$ is $d$-feebly compact; $(iv)$ $\left(\mathscr{I}_\lambda^n,\tau\right)$ is H-closed; $(v)$ $\left(\mathscr{I}_\lambda^n,\tau\right)$ is $\mathbb{N}_{\mathfrak{d}}$-compact for the discrete countable space $\mathbb{N}_{\mathfrak{d}}$; $(vi)$ $\left(\mathscr{I}_\lambda^n,\tau\right)$ is $\mathbb{R}$-compact; $(vii)$ $\left(\mathscr{I}_\lambda^n,\tau\right)$ is  infra H-closed. Also we prove that for an arbitrary positive integer $n$ and an arbitrary infinite cardinal $\lambda$  every shift-continuous semiregular feebly compact $T_1$-topology $\tau$ on $\mathscr{I}_\lambda^n$ is compact.
\end{abstract}

\maketitle

We follow the terminology of~\cite{Carruth-Hildebrant-Koch-1983-1986, Clifford-Preston-1961-1967, Engelking-1989, Petrich-1984, Ruppert-1984}. If $X$ is a topological space and $A\subseteq X$, then by $\operatorname{cl}_X(A)$ and $\operatorname{int}_X(A)$ we denote the topological closure and interior of $A$ in $X$, respectively. By $|A|$ we denote the cardinality of a set $A$, by $A\triangle B$ the symmetric difference of sets $A$ and $B$, by $\mathbb{N}$ the set of positive integers, and by $\omega$ the first infinite cardinal.

A semigroup $S$ is called \emph{inverse} if every $a$ in $S$ possesses an unique inverse $a^{-1}$, i.e. if there exists an unique element $a^{-1}$ in $S$ such that
\begin{equation*}
    aa^{-1}a=a \qquad \mbox{and} \qquad a^{-1}aa^{-1}=a^{-1}.
\end{equation*}
A map which associates to any element of an inverse semigroup its inverse is called the \emph{inversion}.

A {\it topological} ({\it inverse}) {\it semigroup} is a topological space together with a continuous semigroup operation (and an~inversion, respectively). Obviously, the inversion defined on a topological inverse semigroup is a homeomorphism. If $S$ is a~semigroup (an~inverse semigroup) and $\tau$ is a topology on $S$ such that $(S,\tau)$ is a topological (inverse) semigroup, then we
shall call $\tau$ a \emph{semigroup} (\emph{inverse}) \emph{topology} on $S$. A {\it semitopological semigroup} is a topological space together with a separately continuous semigroup operation. If $S$ is a~semigroup (an~inverse semigroup) and $\tau$ is a topology on $S$ such that $(S,\tau)$ is a semitopological semigroup (with continuous inversion), then we shall call $\tau$ a \emph{shift-continuous} (\emph{inverse}) \emph{topology} on $S$.

If $S$ is a~semigroup, then by $E(S)$ we denote the subset of all idempotents of $S$. On  the set of idempotents $E(S)$ there exists a natural partial order: $e\leqslant f$ \emph{if and only if} $ef=fe=e$. A \emph{semilattice} is a commutative semigroup of idempotents. A {\it topological} ({\it semitopological}) {\it semilattice} is a topological space together with a continuous (separately continuous) semilattice operation. If $S$ is a~semilattice and $\tau$ is a topology on $S$ such that $(S,\tau)$ is a topological semilattice, then we shall call $\tau$ a \emph{semilattice} \emph{topology} on $S$.

Every inverse semigroup $S$ admits a partial order:
\begin{equation*}
  a\preccurlyeq b \qquad \hbox{if and only if there exists} \qquad e\in E(S) \quad \hbox{such that} \quad a=eb.
\end{equation*}
We shall say that $\preccurlyeq$ is the \emph{natural partial order} on $S$.

Let $\lambda$ be an arbitrary non-zero cardinal. A map $\alpha$ from a subset $D$ of $\lambda$ into $\lambda$ is called a \emph{partial transformation} of $\lambda$. In this case the set $D$ is called the \emph{domain} of $\alpha$ and is denoted by $\operatorname{dom}\alpha$. The image of an element $x\in\operatorname{dom}\alpha$ under $\alpha$ is denoted by $x\alpha$  Also, the set $\{ x\in \lambda\colon y\alpha=x \mbox{ for some } y\in Y\}$ is called the \emph{range} of $\alpha$ and is denoted by $\operatorname{ran}\alpha$. The cardinality of $\operatorname{ran}\alpha$ is called the \emph{rank} of $\alpha$ and is denoted by $\operatorname{rank}\alpha$. For convenience we denote by $\varnothing$ the empty transformation, a partial mapping with $\operatorname{dom}\varnothing=\operatorname{ran}\varnothing=\varnothing$.

Let $\mathscr{I}_\lambda$ denote the set of all partial one-to-one transformations of $\lambda$ together with the following semigroup operation:
\begin{equation*}
    x(\alpha\beta)=(x\alpha)\beta \quad \mbox{if} \quad
    x\in\operatorname{dom}(\alpha\beta)=\{
    y\in\operatorname{dom}\alpha\colon
    y\alpha\in\operatorname{dom}\beta\}, \qquad \mbox{for} \quad
    \alpha,\beta\in\mathscr{I}_\lambda.
\end{equation*}
The semigroup $\mathscr{I}_\lambda$ is called the \emph{symmetric
inverse semigroup} over the cardinal $\lambda$~(see \cite{Clifford-Preston-1961-1967}). The symmetric
inverse semigroup was introduced by V.~V.~Wagner~\cite{Wagner-1952}
and it plays a major role in the theory of semigroups.


Put
$\mathscr{I}_\lambda^n=\{ \alpha\in\mathscr{I}_\lambda\colon
\operatorname{rank}\alpha\leqslant n\}$,
for $n=1,2,3,\ldots$. Obviously,
$\mathscr{I}_\lambda^n$ ($n=1,2,3,\ldots$) are inverse semigroups,
$\mathscr{I}_\lambda^n$ is an ideal of $\mathscr{I}_\lambda$, for each $n=1,2,3,\ldots$. The semigroup
$\mathscr{I}_\lambda^n$ is called the \emph{symmetric inverse semigroup of
finite transformations of the rank $\leqslant n$}. By
\begin{equation*}
\left({%
\begin{array}{cccc}
  x_1 & x_2 & \cdots & x_n \\
  y_1 & y_2 & \cdots & y_n \\
\end{array}%
}\right)
\end{equation*}
we denote a partial one-to-one transformation which maps $x_1$ onto $y_1$, $x_2$ onto $y_2$, $\ldots$, and $x_n$ onto $y_n$. Obviously, in such case we have $x_i\neq x_j$ and $y_i\neq y_j$ for $i\neq j$ ($i,j=1,2,3,\ldots,n$). The empty partial map $\varnothing\colon \lambda\rightharpoonup\lambda$ is denoted by $\operatorname{\textbf{0}}$. It is obvious that $\operatorname{\textbf{0}}$ is zero of the semigroup $\mathscr{I}_\lambda^n$.

Let $\lambda$ be a non-zero cardinal. On the set
 $
 B_{\lambda}=(\lambda\times\lambda)\cup\{ 0\}
 $,
where $0\notin\lambda\times\lambda$, we define the semigroup
operation ``$\, \cdot\, $'' as follows
\begin{equation*}
(a, b)\cdot(c, d)=
\left\{
  \begin{array}{cl}
    (a, d), & \hbox{ if~ } b=c;\\
    0, & \hbox{ if~ } b\neq c,
  \end{array}
\right.
\end{equation*}
and $(a, b)\cdot 0=0\cdot(a, b)=0\cdot 0=0$ for $a,b,c,d\in
\lambda$. The semigroup $B_{\lambda}$ is called the
\emph{semigroup of $\lambda\times\lambda$-matrix units}~(see
\cite{Clifford-Preston-1961-1967}). Obviously, for any cardinal $\lambda>0$, the semigroup
of $\lambda\times\lambda$-matrix units $B_{\lambda}$ is isomorphic
to $\mathscr{I}_\lambda^1$.

A subset $A$ of a topological space $X$ is called \emph{regular open} if $\operatorname{int}_X(\operatorname{cl}_X(A))=A$.

We recall that a topological space $X$ is said to be
\begin{itemize}
  \item \emph{functionally Hausdorff} if for every pair of distinct points $x_1,x_2\in X$ there exists a continuous function $f\colon X\rightarrow [0,1]$ such that $f(x_1)=0$ and $f(x_2)=1$;
  \item \emph{semiregular} if $X$ has a base consisting of regular open subsets;
  \item \emph{quasiregular} if for any non-empty open set $U\subset X$ there exists a non-empty open set $V\subset U$ such that $\operatorname{cl}_X(V) \subseteq U$;
  \item \emph{compact} if each open cover of $X$ has a finite subcover;
  \item \emph{sequentially compact} if each sequence $\{x_i\}_{i\in\mathbb{N}}$ of $X$ has a convergent subsequence in $X$;
  \item \emph{countably compact} if each open countable cover of $X$ has a finite subcover;
  \item \emph{H-closed} if $X$ is a closed subspace of every Hausdorff topological space in which it is contained;
  \item \emph{infra H-closed} provided that any continuous image of $X$ into any first countable Hausdorff space is closed (see \cite{Hajek-Todd-1975});
  \item \emph{countably compact at a subset} $A\subseteq X$ if every infinite subset $B\subseteq A$  has  an  accumulation  point $x$ in $X$;
  \item \emph{countably pracompact} if there exists a dense subset $A$ in $X$  such that $X$ is countably compact at $A$;
  \item \emph{feebly compact} if each locally finite open cover of $X$ is finite;
  \item $d$-\emph{feebly compact} (or \emph{\textsf{DFCC}}) if every discrete family of open subsets in $X$ is finite (see \cite{Matveev-1998});
  \item \emph{pseudocompact} if $X$ is Tychonoff and each continuous real-valued function on $X$ is bounded;
  \item $Y$-\emph{compact} for some topological space $Y$, if $f(X)$ is compact, for any continuous map $f\colon X\to Y$.
\end{itemize}

According to Theorem~3.10.22 of \cite{Engelking-1989}, a Tychonoff topological space $X$ is feebly compact if and only if $X$ is pseudocompact. Also, a Hausdorff topological space $X$ is feebly compact if and only if every locally finite family of non-empty open subsets of $X$ is finite.  Every compact space and every sequentially compact space are countably compact, every countably compact space is countably pracompact, every countably pracompact space is feebly compact (see \cite{Arkhangelskii-1992}), every H-closed space is feebly compact too (see \cite{Gutik-Ravsky-2015a}). Also, every space feebly compact is infra H-closed by Proposition 2 and Theorem 3 of \cite{Hajek-Todd-1975}.

Topological properties of an infinite (semi)topological semigroup $\lambda\times \lambda$-matrix units were studied in \cite{Gutik-Pavlyk-2005, Gutik-Pavlyk-2005a, Gutik-Pavlyk-Reiter-2009}. In \cite{Gutik-Pavlyk-2005a} it was shown that on the infinite semitopological semigroup $\lambda\times \lambda$-matrix units $B_\lambda$ there exists a unique Hausdorff topology $\tau_c$ such that $(B_\lambda,\tau_c)$ is a compact semitopological semigroup and it was also shown that every pseudocompact Hausdorff shift-continuous topology $\tau$ on $B_\lambda$ is compact. Also, in \cite{Gutik-Pavlyk-2005a} it was proved that every non-zero element of a Hausdorff semitopological semigroup $\lambda\times \lambda$-matrix units $B_\lambda$ is an isolated point in the topological space $B_\lambda$. In \cite{Gutik-Pavlyk-2005} it was shown that the infinite semigroup $\lambda\times \lambda$-matrix units $B_\lambda$ cannot be embedded into a compact Hausdorff topological semigroup, every Hausdorff topological inverse semigroup $S$ that contains $B_\lambda$ as a subsemigroup, contains $B_\lambda$ as a closed subsemigroup, i.e., $B_\lambda$ is \emph{algebraically complete} in the class of Hausdorff topological inverse semigroups. This result in \cite{Gutik-Lawson-Repov-2009} was extended onto so called inverse semigroups with \emph{tight ideal series} and, as a corollary, onto the semigroup $\mathscr{I}_\lambda^n$. Also, in \cite{Gutik-Reiter-2009} it was proved that for every positive integer $n$ the semigroup $\mathscr{I}_\lambda^n$ is \emph{algebraically $h$-complete} in the class of Hausdorff topological inverse semigroups, i.e., every homomorphic image of $\mathscr{I}_\lambda^n$ is algebraically complete in the class of Hausdorff topological inverse semigroups. In the paper \cite{Gutik-Reiter-2010} this result was extended onto the class of Hausdorff semitopological inverse semigroups and it was shown therein that for an infinite cardinal $\lambda$ the semigroup $\mathscr{I}_\lambda^n$ admits a unique Hausdorff topology $\tau_c$ such that $(\mathscr{I}_\lambda^n,\tau_c)$ is a compact semitopological semigroup. Also, it was proved in \cite{Gutik-Reiter-2010} that every countably compact Hausdorff shift-continuous topology $\tau$ on $B_\lambda$ is compact. In \cite{Gutik-Pavlyk-Reiter-2009} it was shown that a topological semigroup of finite partial bijections $\mathscr{I}_\lambda^n$ with a compact subsemigroup of idempotents is absolutely H-closed (i.e., every homomorphic image of $\mathscr{I}_\lambda^n$ is algebraically complete in the class of Hausdorff topological semigroups) and any
countably compact topological semigroup does not contain $\mathscr{I}_\lambda^n$ as a subsemigroup for infinite cardinal $\lambda$. In \cite{Gutik-Pavlyk-Reiter-2009} there were given sufficient conditions onto a topological semigroup $\mathscr{I}_\lambda^1$ to be non-H-closed. Also in \cite{Gutik-2014} it was proved that an infinite semitopological semigroup of $\lambda\times\lambda$-matrix units $B_\lambda$ is H-closed in the class of semitopological semigroups if and only if the space $B_\lambda$ is compact.

For an arbitrary positive integer $n$ and an arbitrary non-zero cardinal $\lambda$ we put
\begin{equation*}
  \exp_n\lambda=\left\{A\subseteq \lambda\colon |A|\leqslant n\right\}.
\end{equation*}

It is obvious that for any positive integer $n$ and any non-zero cardinal $\lambda$ the set $\exp_n\lambda$ with the binary operation $\cap$ is a semilattice. Later in this paper by $\exp_n\lambda$ we shall denote the semilattice $\left(\exp_n\lambda,\cap\right)$. It is easy to see that $\exp_n\lambda$ is isomorphic to the subsemigroup of idempotents (the band) of the semigroup $\mathscr{I}_\lambda^n$ for any positive integer $n$. We observe that for every positive integer $n$ the band of the semigroup $\mathscr{I}_\lambda^n$ is isomorphic to the semilattice $\exp_n\lambda$ by the mapping $E(\mathscr{I}_\lambda^n)\ni\varepsilon\mapsto\operatorname{dom}\varepsilon$.

In the paper \cite{Gutik-Sobol-2016} feebly compact shift-continuous topologies $\tau$ on the semilattice $\exp_n\lambda$ were studied, and all compact semilattice topologies on $\exp_n\lambda$ were described. In \cite{Gutik-Sobol-2016} it was whown that for an arbitrary positive integer $n$ and an arbitrary infinite cardinal $\lambda$ for a $T_1$-topology $\tau$ on $\exp_n\lambda$ the following conditions are equivalent: $(i)$ $\left(\exp_n\lambda,\tau\right)$ is a compact topological semilattice; $(ii)$ $\left(\exp_n\lambda,\tau\right)$ is a countably compact topological semilattice; $(iii)$ $\left(\exp_n\lambda,\tau\right)$ is a feebly compact topological semilattice; $(iv)$ $\left(\exp_n\lambda,\tau\right)$ is a compact semitopological semilattice; $(v)$ $\left(\exp_n\lambda,\tau\right)$ is a countably compact semitopological semilattice. Also, in \cite{Gutik-Sobol-2016} there was constructed a countably pracompact H-closed quasiregular non-semiregular topology $\tau_{\operatorname{\textsf{fc}}}^2$ such that $\left(\exp_2\lambda,\tau_{\operatorname{\textsf{fc}}}^2\right)$ is a semitopological semilattice with the discontinuous semilattice operation and it was proved that for an arbitrary positive integer $n$ and an arbitrary infinite cardinal $\lambda$ a semiregular feebly compact semitopological semilattice $\exp_n\lambda$ is a compact topological semilattice. In \cite{Gutik-Sobol-2016a} it was shown that for an arbitrary positive integer $n$ and an arbitrary infinite cardinal $\lambda$ for a $T_1$-topology $\tau$ on $\exp_n\lambda$ the following conditions are equivalent: $(i)$ $\tau$ is countably pracompact; $(ii)$ $\tau$ is feebly compact; $(iii)$ $\tau$ is $d$-feebly compact; $(iv)$ $\left(\exp_n\lambda,\tau\right)$ is an H-closed space.

\smallskip

This paper is a continuation of \cite{Gutik-Lawson-Repov-2009, Gutik-Pavlyk-2005a, Gutik-Reiter-2009, Gutik-Reiter-2010}. We study feebly compact shift-continuous $T_1$-topologies on the semigroup $\mathscr{I}_\lambda^n$. For any positive integer $n\geqslant2$ and any infinite cardinal $\lambda$ a Hausdorff countably pracompact non-compact shift-continuous topology on $\mathscr{I}_\lambda^n$ is constructed. We show that for an arbitrary positive integer $n$ and an arbitrary infinite cardinal $\lambda$ for a $T_1$-topology $\tau$ on $\mathscr{I}_\lambda^n$ the following conditions are equivalent: $(i)$ $\tau$ is countably pracompact; $(ii)$ $\tau$ is feebly compact; $(iii)$ $\tau$ is $d$-feebly compact; $(iv)$ $\left(\mathscr{I}_\lambda^n,\tau\right)$ is H-closed; $(v)$ $\left(\mathscr{I}_\lambda^n,\tau\right)$ is $\mathbb{N}_{\mathfrak{d}}$-compact for the discrete countable space $\mathbb{N}_{\mathfrak{d}}$;
$(vi)$ $\left(\mathscr{I}_\lambda^n,\tau\right)$ is $\mathbb{R}$-compact; $(vii)$ $\left(\mathscr{I}_\lambda^n,\tau\right)$ is  infra H-closed. Also we prove that for an arbitrary positive integer $n$ and an arbitrary infinite cardinal $\lambda$  every shift-continuous semiregular feebly compact $T_1$-topology $\tau$ on $\mathscr{I}_\lambda^n$ is compact.



\medskip

Later we shall assume that $n$ is an arbitrary positive integer.

For every element $\alpha$ of the semigroup $\mathscr{I}_\lambda^n$ we put
\begin{equation*}
  {\uparrow}_l\alpha=\left\{\beta\in\mathscr{I}_\lambda^n\colon \alpha\alpha^{-1}\beta=\alpha\right\} \qquad \hbox{and} \qquad
  {\uparrow}_r\alpha=\left\{\beta\in\mathscr{I}_\lambda^n\colon\beta\alpha^{-1}\alpha=\alpha\right\}.
\end{equation*}
Then  Proposition~5 of \cite{Gutik-Reiter-2010} implies that ${\uparrow}_l\alpha={\uparrow}_r\alpha$ and by Lemma~6 of \cite[Section~1.4]{Lawson-1998} we have that $\alpha\preccurlyeq\beta$ if and only if $\beta\in{\uparrow}_l\alpha$ for $\alpha,\beta\in\mathscr{I}_\lambda^n$. Hence we put ${\uparrow}_{\preccurlyeq}\alpha={\uparrow}_l\alpha={\uparrow}_r\alpha$ for any $\alpha\in\mathscr{I}_\lambda^n$.

The definition of the semigroup operation of $\mathscr{I}_\lambda^n$ implies the following trivial lemma.

\begin{lemma}\label{lemma-2.1}
Let $n$ be an arbitrary positive integer and $\lambda$ be any cardinal. Then for any elements $\alpha$ and $\beta$ of the semigroup $\mathscr{I}_\lambda^n$ the sets $\alpha\mathscr{I}_\lambda^n\beta$ and
\begin{equation*}
  {\downarrow}_\preccurlyeq\alpha=\left\{\gamma\in\mathscr{I}_\lambda^n \colon \gamma\preccurlyeq\alpha\right\}
\end{equation*}
are finite.
\end{lemma}

\begin{proof}
For any elements $\alpha$ and $\beta$ of $\mathscr{I}_\lambda^n$ we have that
\begin{equation*}
  \alpha\mathscr{I}_\lambda^n\beta=\alpha\mathscr{I}_\lambda^n\cap\mathscr{I}_\lambda^n\beta=\left\{\gamma\in\mathscr{I}_\lambda^n \colon \operatorname{dom}\gamma\subseteq\operatorname{dom}\alpha \; \hbox{~and~} \;  \operatorname{ran}\gamma\subseteq\operatorname{ran}\beta\right\}.
\end{equation*}
Since the sets $\operatorname{dom}\alpha$ and $\operatorname{ran}\beta$ are finite, $\alpha\mathscr{I}_\lambda^n\beta$ is finite, as well.

For every $\gamma\in{\downarrow}_\preccurlyeq\alpha$ the definition of the natural partial order $\preccurlyeq$ on the semigroup $\mathscr{I}_\lambda^n$ (see \cite[Chapter~1]{Lawson-1998}) implies that the finite partial map $\gamma$ is a restriction of the finite partial map $\alpha$ onto the set $A=\operatorname{dom}\alpha\cap\operatorname{dom}\varepsilon$, where $\varepsilon$ is an idempotent of $\mathscr{I}_\lambda^n$ such that $\gamma=\varepsilon\alpha$. This implies that the set ${\downarrow}_\preccurlyeq\alpha$ is finite.
\end{proof}

\begin{lemma}\label{lemma-2.2}
Let $n$ be an arbitrary positive integer, $\lambda$ be any infinite cardinal and $\tau$ be a shift-continuous $T_1$-topology on semigroup $\mathscr{I}_\lambda^n$. Then for every element $\alpha$ of the semigroup $\mathscr{I}_\lambda^n$ the set ${\uparrow}_{\preccurlyeq}\alpha$ is open-and-closed in $\left(\mathscr{I}_\lambda^n,\tau\right)$, the space $\left(\mathscr{I}_\lambda^n,\tau\right)$ is functionally Hausdorff and hence it is quasi-regular.
\end{lemma}

\begin{proof}
Fix an arbitrary $\alpha\in\mathscr{I}_\lambda^n$. Then $\alpha\in\alpha\mathscr{I}_\lambda^n\alpha$ and
 \begin{equation*}
\alpha\mathscr{I}_\lambda^n\alpha=\alpha\mathscr{I}_\lambda^n\cap\mathscr{I}_\lambda^n\alpha= \alpha\alpha^{-1}\mathscr{I}_\lambda^n\cap\mathscr{I}_\lambda^n\alpha^{-1}\alpha=\alpha\alpha^{-1}\mathscr{I}_\lambda^n\alpha^{-1}\alpha,
\end{equation*}
because $\mathscr{I}_\lambda^n$ is an inverse semigroup. Since the topology $\tau$ is $T_1$, Lemma~\ref{lemma-2.1} implies that the set $\left(\alpha\mathscr{I}_\lambda^n\alpha\right) \setminus\{\alpha\}$ is closed in $\left(\mathscr{I}_\lambda^n,\tau\right)$. By the separate continuity of the semigroup operation in $\left(\mathscr{I}_\lambda^n,\tau\right)$ we have that there exists an open neighbourhood $U(\alpha)$ of the point $\alpha$ in $\left(\mathscr{I}_\lambda^n,\tau\right)$ such that
\begin{equation*}
\alpha\alpha^{-1}\cdot U(\alpha)\cdot\alpha^{-1}\alpha\subseteq\mathscr{I}_\lambda^n\setminus\left( \left(\alpha\mathscr{I}_\lambda^n\cup\mathscr{I}_\lambda^n\alpha\right) \setminus\{\alpha\}\right).
\end{equation*}
The last inclusion implies that $U(\alpha)\subseteq {\uparrow}\alpha$. Again, since the semigroup operation in $\left(\mathscr{I}_\lambda^n,\tau\right)$ is separately continuous the set ${\uparrow}_{\preccurlyeq}\alpha$ is open in $\left(\mathscr{I}_\lambda^n,\tau\right)$ as a full preimage of $U(\alpha)$ and the set ${\uparrow}_{\preccurlyeq}\alpha$ is closed in $\left(\mathscr{I}_\lambda^n,\tau\right)$ as a full preimage of the singleton set $\{\alpha\}$.

Fix arbitrary distinct elements $\alpha$ and $\beta$ of the semigroup $\mathscr{I}_\lambda^n$. Then either $\alpha$ and $\beta$ are comparable or not with  respect to the natural partial order on $\mathscr{I}_\lambda^n$. If $\alpha\preccurlyeq\beta$ or $\alpha$ and $\beta$ are incomparable in $\left(\mathscr{I}_\lambda^n,\preccurlyeq\right)$ then it is obvious that the map $g\colon \mathscr{I}_\lambda^n\to[0,1]$ defined by the formula
\begin{equation*}
  (\gamma)f=
\left\{
  \begin{array}{ll}
    1, & \hbox{if~} \gamma\in{\uparrow}_{\preccurlyeq}\beta;\\
    0, & \hbox{if~} \gamma\notin{\uparrow}_{\preccurlyeq}\beta
  \end{array}
\right.
\end{equation*}
is continuous. We observe that quasi-regularity of $\left(\mathscr{I}_\lambda^n,\tau\right)$ follows from the fact that every non-empty open subset $U$ of $\left(\mathscr{I}_\lambda^n,\tau\right)$ contains a maximal element $\delta$ with respect to the natural partial order $\preccurlyeq$ on $\mathscr{I}_\lambda^n$ such that ${\uparrow}_{\preccurlyeq}\alpha$ is an open-and-closed subset of $\left(\mathscr{I}_\lambda^n,\tau\right)$ and hence, since $\tau$ is a $T_1$-topology, $\{\alpha\}\subseteq U$ is an open-and-closed subset of $\left(\mathscr{I}_\lambda^n,\tau\right)$.
\end{proof}

A topological space $X$ is called
\begin{itemize}
  \item \emph{totally disconnected} if the connected components in $X$ are singleton sets;
  \item \emph{scattered} if $X$ does not contain non-empty dense in itself subset, which is equivalent that every non-empty subset of $X$ has an isolated point in itself.
\end{itemize}

Lemma~\ref{lemma-2.2} implies the following corollary:

\begin{corollary}\label{corollary-2.2a}
Let $n$ be an arbitrary positive integer, $\lambda$ be any infinite cardinal and $\tau$ be a shift-continuous $T_1$-topology on the semigroup $\mathscr{I}_\lambda^n$. Then $\left(\mathscr{I}_\lambda^n,\tau\right)$ is a totally disconnected scattered space.
\end{corollary}

A partial order $\leq$ on a topological space $X$ is called closed if  the relation $\leq$ is a closed subset of $X\times X$ in the product topology. In this case $(X,\leq)$ is called a \emph{pospace}~\cite{Gierz-Hofmann-Keimel-Lawson-Mislove-Scott-2003}.

Lemma~\ref{lemma-2.2} and Proposition~VI-1.4 from~\cite{Gierz-Hofmann-Keimel-Lawson-Mislove-Scott-2003} imply the following corollary:

\begin{corollary}\label{corollary-2.2b}
Let $n$ be an arbitrary positive integer, $\lambda$ be any infinite cardinal and $\tau$ be a shift-continuous $T_1$-topology on semigroup $\mathscr{I}_\lambda^n$. Then $\left(\mathscr{I}_\lambda^n,\tau,\preccurlyeq\right)$ is a pospace
\end{corollary}

The following example shows that the statement of Lemma~\ref{lemma-2.2} does not hold in the case when $\left(\mathscr{I}_\lambda^n,\tau\right)$ is a $T_0$-space.

\begin{example}\label{example-2.3}
For an arbitrary positive integer $n$ and an arbitrary infinite cardinal $\lambda$ we define a topology $\tau_0$ on $\mathscr{I}_\lambda^n$ in the following way:
\begin{itemize}
  \item[$(i)$] all non-zero elements of the semigroup $\mathscr{I}_\lambda^n$ are isolated points in $\left(\mathscr{I}_\lambda^n,\tau_0\right)$; \; and
  \item[$(ii)$] $\mathscr{I}_\lambda^n$ is the unique open neighbourhood of zero in $\left(\mathscr{I}_\lambda^n,\tau_0\right)$.
\end{itemize}
Simple verifications show that the semigroup operation and inversion on $\left(\mathscr{I}_\lambda^n,\tau_0\right)$ are continuous.
\end{example}

We need the following example from \cite{Gutik-Reiter-2010}.

\begin{example}[\cite{Gutik-Reiter-2010}]\label{example-2.4}
Fix an arbitrary positive integer $n$. The following family
\begin{equation*}
\begin{split}
  \mathscr{B}_{\operatorname{\textsf{c}}}& =\left\{U_\alpha(\alpha_1,\ldots,\alpha_k)=
    {\uparrow}_{\preccurlyeq}\alpha\setminus({\uparrow}_{\preccurlyeq}\alpha_1\cup\cdots\cup
    {\uparrow}_{\preccurlyeq}\alpha_k) \colon\right. \\
    & \qquad \left. \alpha_i\in{\uparrow}_{\preccurlyeq}\alpha\setminus
    \{\alpha\}, \alpha, \alpha_i\in\mathscr{I}_\lambda^n, i=1,\ldots, k\right\}
\end{split}
\end{equation*}
determines a base of the topology $\tau_{\operatorname{\textsf{c}}}$ on $\mathscr{I}_\lambda^n$. By Proposition~10 from \cite{Gutik-Reiter-2010}, $\left(\mathscr{I}_\lambda^n,\tau_{\operatorname{\textsf{c}}}\right)$ is a Hausdorff compact semitopological semigroup with continuous inversion.
\end{example}

By Theorem~7 from \cite{Gutik-Reiter-2010}, for an arbitrary infinite cardinal $\lambda$ and any positive integer $n$ every countably compact Hausdorff semitopological semigroup $\mathscr{I}_\lambda^n$ is topologically isomorphic to $\left(\mathscr{I}_\lambda^n,\tau_{\operatorname{\textsf{c}}}\right)$.  By Corollary~\ref{corollary-2.2a} the topological space $\left(\mathscr{I}_\lambda^n,\tau_{\operatorname{\textsf{c}}}\right)$ is scattered. Since every countably compact scattered $T_3$-space is sequentially compact (see \cite[Theorem~5.7]{Vaughan-1984}), $\left(\mathscr{I}_\lambda^n,\tau_{\operatorname{\textsf{c}}}\right)$ is a sequentially compact space.

Next we summarise the above results in the following theorem.

\begin{theorem}\label{theorem-2.5}
Let $n$ be an arbitrary positive integer, $\lambda$ be any infinite cardinal and $\tau$ be a $T_1$-shift continuous topology on the semigroup $\mathscr{I}_\lambda^n$. Then the following conditions are equivalent:
\begin{itemize}
  \item[$(i)$] $\tau$ is compact;
  \item[$(ii)$] $\tau=\tau_{\operatorname{\textsf{c}}}$;
  \item[$(iii)$] $\tau$ is countably compact;
  \item[$(iv)$] $\tau$ is sequentially compact.
\end{itemize}
\end{theorem}

Since every feebly compact Hausdorff topology on the semigroup $\mathscr{I}_\lambda^1$ is compact, it is natural to ask: \emph{Does there exist a shift-continuous Hausdorff non-compact feebly compact topology $\tau$ on the semigroup $\mathscr{I}_\lambda^n$ for $n\geqslant 2$?}

The following example shows that for any infinite cardinal $\lambda$ and any positive integer $n\geqslant 2$ there exists a Hausdorff feebly compact topology $\tau$ on the semigroup $\mathscr{I}_\lambda^n$ such that $\left(\mathscr{I}_\lambda^n,\tau\right)$ is a non-compact semitopological semigroup.

\begin{example}\label{example-2.6}
Let $\lambda$ be any infinite cardinal and $\tau_{\operatorname{\textsf{c}}}^2=\tau_{\operatorname{\textsf{c}}}$ be the topology on the semigroup $\mathscr{I}_\lambda^2$ which is defined in Example~\ref{example-2.4}. We construct a  stronger topology $\tau_{\operatorname{\textsf{fc}}}^2$ on $\mathscr{I}_\lambda^2$ then $\tau_{\operatorname{\textsf{c}}}^2$ in the following way. By $\pi\colon\lambda\to\mathscr{I}_\lambda^2\colon a\mapsto\varepsilon_a$ we denote the map which assigns to any element $a\in\lambda$ the identity partial map $\varepsilon_a\colon\{a\}\to\{a\}$. Fix an arbitrary infinite subset $A$ of $\lambda$. For every non-zero element $x\in\mathscr{I}_\lambda^2$ we assume that the base $\mathscr{B}_{\operatorname{\textsf{fc}}}^2(x)$ of the topology $\tau_{\operatorname{\textsf{fc}}}^2$ at the point $x$ coincides with the base of the topology $\tau_{\operatorname{\textsf{c}}}^2$ at $x$, and
\begin{equation*}
\begin{split}
  \mathscr{B}_{\operatorname{\textsf{fc}}}^2(0) & =\left\{U_B(\operatorname{\textbf{0}})=U(\operatorname{\textbf{0}})\setminus\left((B)\pi\cup \left\{\alpha_1,\ldots,\alpha_s\right\}\right)\colon U(0)\in\mathscr{B}_{\operatorname{\textsf{c}}}^2(0), \alpha_1,\ldots,\alpha_s\in\mathscr{I}_\lambda^2\setminus\{\operatorname{\textbf{0}}\} \right.\\
    & \quad \left. \hbox{~and~} B\subseteq\lambda \hbox{~such that~} \left|A\triangle B\right|<\infty \right\}
\end{split}
\end{equation*}
form a base of the topology $\tau_{\operatorname{\textsf{fc}}}^2$ at zero $\operatorname{\textbf{0}}$ of the semigroup $\mathscr{I}_\lambda^2$. Simple verifications show that the family $\left\{\mathscr{B}_{\operatorname{\textsf{fc}}}^2(x)\colon x\in\mathscr{I}_\lambda^2\right\}$ satisfies conditions \textbf{(BP1)--(BP4)} of \cite{Engelking-1989}, and hence $\tau_{\operatorname{\textsf{fc}}}^2$ is a Hausdorff topology on $\mathscr{I}_\lambda^2$.
\end{example}

\begin{proposition}\label{proposition-2.7}
Let $\lambda$ be an arbitrary infinite cardinal. Then \emph{$\left(\mathscr{I}_\lambda^2,\tau_{\operatorname{\textsf{fc}}}^2\right)$} is a countably pracompact semitopological semigroup with continuous inversion.
\end{proposition}

\begin{proof}
It is obvious that the inversion in $\left(\mathscr{I}_\lambda^2,\tau_{\operatorname{\textsf{fc}}}^2\right)$ is continuous and later we shall show that all translations in $\left(\mathscr{I}_\lambda^2,\tau_{\operatorname{\textsf{fc}}}^2\right)$ are continuous maps. We consider the following possible cases.

\textbf{(1)} $\operatorname{\textbf{0}}\cdot \operatorname{\textbf{0}}=\operatorname{\textbf{0}}$. For every basic open neighbourhood $U_B(\operatorname{\textbf{0}})$ of zero in $\left(\mathscr{I}_\lambda^2,\tau_{\operatorname{\textsf{fc}}}^2\right)$ we have that
\begin{equation*}
  U_B(\operatorname{\textbf{0}})\cdot \operatorname{\textbf{0}}=\operatorname{\textbf{0}}\cdot\, U_B(\operatorname{\textbf{0}})= \{\operatorname{\textbf{0}}\}\subset  U_\pi(\operatorname{\textbf{0}}).
\end{equation*}

\textbf{(2)} $\alpha\cdot \operatorname{\textbf{0}}=\operatorname{\textbf{0}}$. For all basic open neighbourhoods $U_B(\operatorname{\textbf{0}})$ and $U_\alpha(\beta_1,\ldots,\beta_k)$ of zero and an element $\alpha\neq\operatorname{\textbf{0}}$ in $\left(\mathscr{I}_\lambda^2,\tau_{\operatorname{\textsf{fc}}}^2\right)$, respectively, we have that
\begin{equation*}
  U_\alpha(\beta_1,\ldots,\beta_k) \cdot \operatorname{\textbf{0}}=\{\operatorname{\textbf{0}}\}\subset U_B(\operatorname{\textbf{0}}).
\end{equation*}
Let $V_B(\operatorname{\textbf{0}})= \mathscr{I}_\lambda^2\setminus({\uparrow}_{\preccurlyeq}\alpha_1\cup\cdots\cup{\uparrow}_{\preccurlyeq}\alpha_k\cup(B)\pi)$ be an arbitrary basic neighbourhood of zero in $\left(\mathscr{I}_\lambda^2,\tau_{\operatorname{\textsf{fc}}}^2\right)$. Without loss of generality we may assume that
\begin{equation*}
\operatorname{rank}\alpha_1=\ldots=\operatorname{rank}\alpha_k=1\leqslant\operatorname{rank}\alpha.
\end{equation*}
 Put
\begin{equation*}
\begin{split}
  \mathcal{C}_l=\left\{\gamma\in\mathscr{I}_\lambda^2 \right. \colon & \operatorname{rank}\gamma=1 \; \hbox{such that}\; \alpha\gamma=\alpha_i \hbox{~for some~} i=1,\ldots,k \\
    & \left. \hbox{or} \; \alpha\gamma\in E(\mathscr{I}_\lambda^2)\setminus\{\operatorname{\textbf{0}}\}\right\}.
\end{split}
\end{equation*}
The definition of the semigroup $\mathscr{I}_\lambda^2$ implies that the set $\mathcal{C}_l$ is finite. Then we have that $\alpha\cdot W_B(0)\subseteq V_B(0)$ for $W_B(0)=\mathscr{I}_\lambda^2\setminus\bigcup\left\{{\uparrow}_{\preccurlyeq}\gamma\colon \gamma\in\mathcal{C}_l\right\}$.

\textbf{(3)} $ \operatorname{\textbf{0}}\cdot\alpha=\operatorname{\textbf{0}}$. For all basic open neighbourhoods $U_B(\operatorname{\textbf{0}})$ and $U_\alpha(\beta_1,\ldots,\beta_k)$ of zero and an element $\alpha\neq\operatorname{\textbf{0}}$ in $\left(\mathscr{I}_\lambda^2,\tau_{\operatorname{\textsf{fc}}}^2\right)$, respectively, we have that
\begin{equation*}
  \operatorname{\textbf{0}}\cdot U_\alpha(\beta_1,\ldots,\beta_k)=\{\operatorname{\textbf{0}}\}\subset U_B(\operatorname{\textbf{0}}).
\end{equation*}
Let $V_B(\operatorname{\textbf{0}})= \mathscr{I}_\lambda^2\setminus({\uparrow}_{\preccurlyeq}\alpha_1\cup\cdots\cup{\uparrow}_{\preccurlyeq}\alpha_k\cup(B)\pi)$ be an arbitrary basic neighbourhood of zero in $\left(\mathscr{I}_\lambda^2,\tau_{\operatorname{\textsf{fc}}}^2\right)$. Without loss of generality we may assume that
\begin{equation*}
\operatorname{rank}\alpha_1=\ldots=\operatorname{rank}\alpha_k=1\leqslant\operatorname{rank}\alpha.
\end{equation*}
Put
\begin{equation*}
\begin{split}
  \mathcal{C}_r=\left\{\gamma\in\mathscr{I}_\lambda^2\colon\right. & \operatorname{rank}\gamma=1 \; \hbox{such that}\; \gamma\alpha=\alpha_i \hbox{~for some~} i=1,\ldots,k \; \\
    & \left. \hbox{or} \; \gamma\alpha\in E(\mathscr{I}_\lambda^2)\setminus\{\operatorname{\textbf{0}}\}\right\}.
\end{split}
\end{equation*}
The definition of the semigroup $\mathscr{I}_\lambda^2$ implies that the set $\mathcal{C}_r$ is finite. Then we have that $W_B(0)\cdot \alpha\subseteq V_B(0)$ for $W_B(0)=\mathscr{I}_\lambda^2\setminus\bigcup\left\{{\uparrow}_{\preccurlyeq}\gamma\colon \gamma\in\mathcal{C}_r\right\}$.

\textbf{(4)} $\alpha\cdot\beta=\gamma\neq\textbf{0}$ and $\operatorname{rank}\alpha=\operatorname{rank}\beta=\operatorname{rank}\gamma$, i.e.,
$\operatorname{ran}\alpha=\operatorname{dom}\beta$. Then for any open neighbourhoods $U_{\alpha}(\alpha_1,\ldots,\alpha_k)$,
$U_{\beta}(\beta_1,\ldots,\beta_n)$, $U_{\gamma}(\gamma_1,\ldots,\gamma_m)$ of the points $\alpha,\beta$ and $\gamma$ in $\left(\mathscr{I}_\lambda^2,\tau_{\operatorname{\textsf{fc}}}^2\right)$, respectively, we have that
\begin{equation*}
    U_{\alpha}(\alpha_1,\ldots,\alpha_k)\cdot\beta=
    \alpha\cdot U_{\beta}(\beta_1,\ldots,\beta_n)=
    \{\gamma\}\subseteq
    U_{\gamma}(\gamma_1,\ldots,\gamma_m).
\end{equation*}

\textbf{(5)} $\alpha\cdot\beta=\gamma\neq\textbf{0}$ and $\operatorname{rank}\alpha=\operatorname{rank}\gamma=1 $ and $\operatorname{rank}\beta=2$, i.e.,
$\operatorname{ran}\alpha\subsetneqq\operatorname{dom}\beta$. Then for any open neighbourhoods $U_{\beta}(\beta_1,\ldots,\beta_n)$ and
$U_{\gamma}(\gamma_1,\ldots,\gamma_m)$ of the points $\beta$ and $\gamma$ in $\left(\mathscr{I}_\lambda^2,\tau_{\operatorname{\textsf{fc}}}^2\right)$, respectively, we have that
\begin{equation*}
    \alpha\cdot U_{\beta}(\beta_1,\ldots,\beta_n)=
    \{\gamma\}\subseteq
    U_{\gamma}(\gamma_1,\ldots,\gamma_m).
\end{equation*}
Let $U_{\gamma}(\gamma_1,\ldots,\gamma_k)$ be an arbitrary open neighbourhood of the point $\gamma$ in $\left(\mathscr{I}_\lambda^2,\tau_{\operatorname{\textsf{fc}}}^2\right)$ for some $\gamma_1,\ldots,\gamma_k\in{\uparrow}_{\preccurlyeq}\gamma$, i.e., $\operatorname{rank}\gamma_1=\ldots=\operatorname{rank}\gamma_k=2$. Put
\begin{equation*}
\mathcal{Q}=\left\{\delta\in{\uparrow}_{\preccurlyeq}\alpha\colon \delta\beta\in\left\{\gamma_1,\ldots,\gamma_k\right\}\right\}.
\end{equation*}
The definition of the semigroup $\mathscr{I}_\lambda^2$ implies that the set $\mathcal{Q}$ is finite. Then we have that
\begin{equation*}
  U_{\alpha}(\mathcal{Q})\cdot\beta\subseteq U_{\gamma}(\gamma_1,\ldots,\gamma_k)
\end{equation*}
for $U_{\alpha}(\mathcal{Q})={\uparrow}_{\preccurlyeq}\alpha\setminus\left\{\delta\in{\uparrow}_{\preccurlyeq}\alpha\colon \delta\in\mathcal{Q}\right\}$.

\textbf{(6)}  $\alpha\cdot\beta=\gamma\neq\textbf{0}$ and $\operatorname{rank}\beta=\operatorname{rank}\gamma=1$ and $\operatorname{rank}\alpha=2$,
i.e., $\operatorname{dom}\beta\subsetneqq\operatorname{ran}\alpha$. In this case the proof of separate continuity of the semigroup operation on $\left(\mathscr{I}_\lambda^2,\tau_{\operatorname{\textsf{fc}}}^2\right)$ is dual to case \textbf{(5)}.

\textbf{(7)} $\alpha\cdot\beta=\gamma\neq\textbf{0}$,  $\operatorname{rank}\gamma=1$ and $\operatorname{rank}\alpha=\operatorname{rank}\beta=2$.  Then $\alpha$ and $\beta$ are isolated points in $\left(\mathscr{I}_\lambda^2,\tau_{\operatorname{\textsf{fc}}}^2\right)$ and hence
\begin{equation*}
  \alpha\cdot \beta=\gamma\subseteq U_{\gamma}(\gamma_1,\ldots,\gamma_k),
\end{equation*}
for any basic open neighbourhood $U_{\gamma}(\gamma_1,\ldots,\gamma_k)$ of $\gamma$ in $\left(\mathscr{I}_\lambda^2,\tau_{\operatorname{\textsf{fc}}}^2\right)$.

\textbf{(8)} $\alpha\cdot\beta=\textbf{0}$. Then $\operatorname{dom}\beta\cap\operatorname{ran}\alpha=\varnothing$ and hence
\begin{equation*}
  U_{\alpha}(\alpha_1,\ldots,\alpha_k)\cdot\beta=
    \alpha\cdot U_{\beta}(\beta_1,\ldots,\beta_n)=\left\{\textbf{0}\right\}\subset U_B(\operatorname{\textbf{0}}),
\end{equation*}
for any basic open neighbourhoods $U_{\alpha}(\alpha_1,\ldots,\alpha_k)$, $U_{\beta}(\beta_1,\ldots,\beta_n)$ and $U_B(\operatorname{\textbf{0}})$ of $\alpha$, $\beta$ and zero $\textbf{0}$ in $\left(\mathscr{I}_\lambda^2,\tau_{\operatorname{\textsf{fc}}}^2\right)$, respectivelly.

Thus we have shown that the translations in $\left(\mathscr{I}_\lambda^2,\tau_{\operatorname{\textsf{fc}}}^2\right)$ are continuous maps.

Also, the definition of the topology $\tau_{\operatorname{\textsf{fc}}}^2$ on $\mathscr{I}_\lambda^2$ implies that the set $\mathscr{I}_\lambda^2\setminus\mathscr{I}_\lambda^1$ is dense in $\left(\mathscr{I}_\lambda^2,\tau_{\operatorname{\textsf{fc}}}^2\right)$ and every infinite subset of $\mathscr{I}_\lambda^2\mathscr{I}_\lambda^1$ has an accumulation point in $\left(\mathscr{I}_\lambda^2,\tau_{\operatorname{\textsf{fc}}}^2\right)$, and hence the space $\left(\mathscr{I}_\lambda^2,\tau_{\operatorname{\textsf{fc}}}^2\right)$ is countably pracompact.
\end{proof}

\begin{proposition}\label{proposition-2.8}
Let $n$ be an arbitrary positive integer and $\lambda$ be an arbitrary infinite cardinal. Then for every
$d$-feebly compact shift-continuous $T_1$-topology $\tau$ on $\mathscr{I}_\lambda^n$ the subset
$\mathscr{I}_\lambda^n\setminus\mathscr{I}_\lambda^{n-1}$ is dense in $\left(\mathscr{I}_\lambda^n,\tau\right)$.
\end{proposition}

\begin{proof}
Since every quasi-regular $d$-feebly compact space is feebly compact (see \cite[Theorem~2]{Gutik-Sobol-2016a}), by Lemma~\ref{lemma-2.2}  the topology $\tau$ is feebly compact.

Suppose to the contrary that there exists a feebly compact shift-continuous $T_1$-topology $\tau$ on $\mathscr{I}_\lambda^n$ such that
$\mathscr{I}_\lambda^n\setminus\mathscr{I}_\lambda^{n-1}$ is not dense in $\left(\mathscr{I}_\lambda^n,\tau\right)$. Then there exists a point
$\alpha\in\mathscr{I}_\lambda^{n-1}$ of the space $\left(\mathscr{I}_\lambda^n,\tau\right)$ such that $\alpha\notin\operatorname{cl}_{\mathscr{I}_\lambda^n}\left(\mathscr{I}_\lambda^n\setminus\mathscr{I}_\lambda^{n-1}\right)$. This implies that there exists an open neighbourhood $U(\alpha)$ of $\alpha$ in $\left(\mathscr{I}_\lambda^n,\tau\right)$ such that $U(\alpha)\cap \left(\mathscr{I}_\lambda^n\setminus\mathscr{I}_\lambda^{n-1}\right)=\varnothing$. Lemma~\ref{lemma-2.2} implies that ${\uparrow}_{\preccurlyeq}\alpha$ is an open-and-closed subset of $\left(\mathscr{I}_\lambda^n,\tau\right)$ and hence by Theorem~14 of~\cite{Bagley-Connell-McKnight-Jr-1958}, ${\uparrow}_{\preccurlyeq}\alpha$ is feebly compact. This implies that without loos of generality we may assume that $U(\alpha)\subseteq {\uparrow}_{\preccurlyeq}x\cap\mathscr{I}_\lambda^{n-1}$. By the definition of the semigroup $\mathscr{I}_\lambda^n$ we have that there exists a point $\beta\in U(\alpha)$ such that ${\uparrow}_\preccurlyeq\beta\cap U(\alpha)=\{\beta\}$. Again, by Lemma~\ref{lemma-2.2} we have that ${\uparrow}_{\preccurlyeq}\beta$ is an open-and-closed subset of $\left(\mathscr{I}_\lambda^n,\tau\right)$ and hence by Theorem~14 of~\cite{Bagley-Connell-McKnight-Jr-1958}, ${\uparrow}_{\preccurlyeq}\beta$ is feebly compact. Moreover, our choice implies that $\beta$ is an isolated point in the subspace ${\uparrow}_{\preccurlyeq}\beta$ of $\left(\mathscr{I}_\lambda^n,\tau\right)$.

Suppose that
\begin{equation*}
  \beta=
\left(
  \begin{array}{ccc}
    x_1 & \cdots & x_k \\
    y_1 & \cdots & y_k \\
  \end{array}
\right),
\end{equation*}
for some finite subsets $\left\{x_1,\cdots,x_k\right\}$ and $\left\{y_1,\cdots,y_k\right\}$ of distinct points from $\lambda$. Then the above arguments imply that $k<n$. Put $p=n-k$. Next we fix an arbitrary infinite sequence $\left\{a_i\right\}_{i\in\mathbb{N}}$ of distinct elements of the set $\lambda\setminus\left(\left\{x_1,\cdots,x_k\right\}\cup\left\{y_1,\cdots,y_k\right\}\right)$.

For arbitrary positive integer $j$ we put
\begin{equation*}
  \beta_j=
\left(
  \begin{array}{cccccc}
    x_1 & \cdots & x_k & a_{p(j-1)+1} & \cdots & a_{pj}\\
    y_1 & \cdots & y_k & a_{p(j-1)+1} & \cdots & a_{pj}\\
  \end{array}
\right).
\end{equation*}
Then $\beta_j\in\mathscr{I}_\lambda^n$ for any positive integer $j$. Moreover, we have that $\beta_j\in\mathscr{I}_\lambda^n\setminus\mathscr{I}_\lambda^{n-1}$ and $\beta_j\in{\uparrow}_{\preccurlyeq}\beta$ for any positive integer $j$.

We claim that the set ${\uparrow}_{\preccurlyeq}\gamma\cap\left\{\beta_j\colon j\in\mathbb{N}\right\}$  is finite for any $\gamma\in{\uparrow}_{\preccurlyeq}\beta\setminus\{\beta\}$. Indeed, if the set ${\uparrow}_{\preccurlyeq}\gamma\cap\left\{\beta_j\colon j\in\mathbb{N}\right\}$  is infinite for some $\gamma\in{\uparrow}_{\preccurlyeq}\beta\setminus\{\beta\}$ then $\operatorname{dom}\gamma$ contains infinitely many points of the set $\left\{a_i\colon i\in\mathbb{N}\right\}$, which contradicts that $\gamma\in\mathscr{I}_\lambda^n$.

By Lemma~\ref{lemma-2.2} for every $\gamma\in\mathscr{I}_\lambda^n$ the set ${\uparrow}_{\preccurlyeq}\gamma$ is open in $\left(\mathscr{I}_\lambda^n,\tau\right)$. Then since $\beta$ is an isolated point in ${\uparrow}_{\preccurlyeq}\beta$, our claim implies that the infinite family of isolated points $\mathscr{U}=\left\{\left\{b_j\right\}\colon j\in\mathbb{N}\right\}$ is locally finite in ${\uparrow}_{\preccurlyeq}\beta$, which contradicts that the subspace ${\uparrow}_{\preccurlyeq}\beta$ of $\left(\mathscr{I}_\lambda^n,\tau\right)$ is feebly compact. The obtained contradiction implies the statement of the proposition.
\end{proof}

\begin{remark}\label{remark-2.8a}
The following three examples of topological semigroups of matrix units $(B_\lambda,\tau_{mv})$, $(B_\lambda,\tau_{mh})$ and $(B_\lambda,\tau_{mi})$ from \cite{Gutik-Pavlyk-2005} imply that the converse to Proposition~\ref{proposition-2.8} is not true for any infinite cardinal $\lambda$.
\end{remark}

Later by $\mathbb{N}_{\mathfrak{d}}$ and $\mathbb{R}$ we denote the sets of positive integers with the discrete topology and the real numbers with the usual topology.

\begin{theorem}\label{theorem-2.9}
Let $n$ be an arbitrary positive integer and $\lambda$ be an arbitrary infinite cardinal. Then for every shift-continuous $T_1$-topology $\tau$ on the semigroup $\mathscr{I}_\lambda^n$ the following statements are equivalent:
\begin{itemize}
  \item[$(i)$] $\tau$ is countably pracompact;
  \item[$(ii)$] $\tau$ is feebly compact;
  \item[$(iii)$] $\tau$ is $d$-feebly compact;
  \item[$(iv)$] $\left(\mathscr{I}_\lambda^n,\tau\right)$ is H-closed;
  \item[$(v)$] $\left(\mathscr{I}_\lambda^n,\tau\right)$ is $\mathbb{N}_{\mathfrak{d}}$-compact;
  \item[$(vi)$] $\left(\mathscr{I}_\lambda^n,\tau\right)$ is $\mathbb{R}$-compact;
  \item[$(vii)$] $\left(\mathscr{I}_\lambda^n,\tau\right)$ is  infra H-closed.
\end{itemize}
\end{theorem}

\begin{proof}
Implications $(i)\Rightarrow(ii)$ and $(ii)\Rightarrow(iii)$ are trivial.

$(iii)\Rightarrow(ii)$ Suppose that a space $\left(\mathscr{I}_\lambda^n,\tau\right)$ is $d$-feebly compact. By Lemma~\ref{lemma-2.2} it is quasi-regular. Then by Theorem~1 of \cite{Gutik-Sobol-2016a} every quasiregular $d$-feebly compact space is feebly compact and hence so is $\left(\mathscr{I}_\lambda^n,\tau\right)$.

$(ii)\Rightarrow(i)$ Suppose that a space $\left(\mathscr{I}_\lambda^n,\tau\right)$ is feebly compact. By Lemma~\ref{lemma-2.2} the topological space $\left(\mathscr{I}_\lambda^n,\tau\right)$ is Hausdorff. Then by Lemma~1 of \cite{Gutik-Sobol-2016a} every Hausdorff feebly compact space with a dense discrete subspace is countably pracompact (also see Lemma~4.5 of~\cite{Bardyla-Gutik-2016} or Proposition~1 from~\cite{Arkhangelskii-1985} for Tychonoff spaces) and hence so is $\left(\mathscr{I}_\lambda^n,\tau\right)$.

Implication $(iv)\Rightarrow(ii)$ follows from Proposition~4 of \cite{Gutik-Ravsky-2015a}.

$(ii)\Rightarrow(iv)$ We shall show by induction that if $\tau$ is a shift-continuous feebly compact $T_1$-topology on the semigroup $\mathscr{I}_\lambda^n$ then the subspace ${\uparrow}_{\preccurlyeq}\alpha$ of $\left(\mathscr{I}_\lambda^n,\tau\right)$ is H-closed for any $\alpha\in\mathscr{I}_\lambda^n$.

It is obvious that for any $\alpha\in\mathscr{I}_\lambda^n$ with $\operatorname{rank}\alpha=n$ the set ${\uparrow}_{\preccurlyeq}\alpha=\{\alpha\}$ is singleton, and since $\left(\mathscr{I}_\lambda^n,\tau\right)$ is a $T_1$-space, ${\uparrow}_{\preccurlyeq}\alpha$ is H-closed.

Fix an arbitrary $\alpha\in\mathscr{I}_\lambda^n$ with $\operatorname{rank}\alpha=n-1$. By Lemma~\ref{lemma-2.2}, ${\uparrow}_{\preccurlyeq}\alpha$ is an open-and-closed subset of $\left(\mathscr{I}_\lambda^n,\tau\right)$ and hence by Theorem~14 from \cite{Bagley-Connell-McKnight-Jr-1958} the space ${\uparrow}_{\preccurlyeq}\alpha$ is feebly compact. Since by Lemma~\ref{lemma-2.2} every point $\beta$ of ${\uparrow}_{\preccurlyeq}\alpha$ with $\operatorname{rank}\alpha=n$ is isolated in $\left(\mathscr{I}_\lambda^n,\tau\right)$, the feeble compactness of ${\uparrow}_{\preccurlyeq}\alpha$ implies that $\alpha$ is a non-isolated point of $\left(\mathscr{I}_\lambda^n,\tau\right)$ and the space ${\uparrow}_{\preccurlyeq}\alpha$ is compact. This implies that ${\uparrow}_{\preccurlyeq}\alpha$ is H-closed.

Next we shall prove the following statement: \emph{if for some positive integer $k<n$ for any $\alpha\in\mathscr{I}_\lambda^n$ with $\operatorname{rank}\alpha\leqslant k$ the subspace ${\uparrow}_{\preccurlyeq}\alpha$ is H-closed then ${\uparrow}_{\preccurlyeq}\beta$ is H-closed for any $\beta\in\mathscr{I}_\lambda^n$ with $\operatorname{rank}\beta=k-1$}.

Suppose to the contrary that there exists a shift-continuous feebly compact $T_1$-topology $\tau$ on the semigroup $\mathscr{I}_\lambda^n$ such that for some positive integer $k<n$ for any $\alpha\in\mathscr{I}_\lambda^n$ with $\operatorname{rank}\alpha=k$ the subspace ${\uparrow}_{\preccurlyeq}\alpha$ is H-closed and ${\uparrow}_{\preccurlyeq}\beta$ is not an H-closed space for some $\beta\in\mathscr{I}_\lambda^n$ with $\operatorname{rank}\beta=k-1$. Then there exists a Hausdorff topological space $X$ which contains the space ${\uparrow}_{\preccurlyeq}\beta$ as a dense proper subspace. We observe that by Lemma~\ref{lemma-2.2} and Theorem~14 of \cite{Bagley-Connell-McKnight-Jr-1958} the space ${\uparrow}_{\preccurlyeq}\beta$ is feebly compact.

Fix an arbitrary $x\in X\setminus{\uparrow}_{\preccurlyeq}\beta$. The Hausdorffness of $X$ implies that there exist open neighbourhoods $U_X(x)$ and $U_X(\beta)$ of the points $x$ and $\beta$ in $X$, respectively, such that $U_X(x)\cap U_X(\beta)=\varnothing$. Then the assumption of induction implies that without loss of generality we may assume that there do not exist finitely many $\alpha_1,\ldots,\alpha_m\in {\uparrow}_{\preccurlyeq}\beta$ with $\operatorname{rank}\alpha_1=\ldots=\operatorname{rank}\alpha_m=k$ such that
\begin{equation*}
U_X(x)\cap{\uparrow}_{\preccurlyeq}\beta\subseteq {\uparrow}_{\preccurlyeq}\alpha_1\cup\cdots\cup{\uparrow}_{\preccurlyeq}\alpha_m.
\end{equation*}

Fix an arbitrary $\alpha_1\in {\uparrow}_{\preccurlyeq}\beta$ such that $\operatorname{rank}\alpha_1=k$ and ${\uparrow}_{\preccurlyeq}\alpha_1\cap U_X(x)\neq\varnothing$. Proposition~1.3.1 of \cite{Engelking-1989}, Lemma~\ref{lemma-2.2} and Proposition~\ref{proposition-2.8} imply that there exists $\gamma_1\in\mathscr{I}_\lambda^n\setminus\mathscr{I}_\lambda^{n-1}$ such that $\gamma_1\in{\uparrow}_{\preccurlyeq}\alpha_1\cap U_X(x)$. Next, by induction using Proposition~1.3.1 of \cite{Engelking-1989}, Lemma~\ref{lemma-2.2} and Proposition~\ref{proposition-2.8} we construct sequences $\left\{\alpha_i\right\}_{i\in\mathbb{N}}$ and $\left\{\gamma_i\right\}_{i\in\mathbb{N}}$ of distinct points of the set ${\uparrow}_{\preccurlyeq}\beta$ such that the following conditions hold:
\begin{itemize}
  \item[$(a)$] $\operatorname{rank}\alpha_{i+1}=k$ and ${\uparrow}_{\preccurlyeq}\alpha_{i+1}\setminus\left({\uparrow}_{\preccurlyeq}\alpha_1\cup \cdots\cup {\uparrow}_{\preccurlyeq}\alpha_{i}\right)\cap U_X(x)\neq\varnothing$; \: and
  \item[$(b)$] $\gamma_{i+1}\in\mathscr{I}_\lambda^n\setminus\mathscr{I}_\lambda^{n-1}$ and $\gamma_{i+1}\in {\uparrow}_{\preccurlyeq}\alpha_{i+1}\setminus\left({\uparrow}_{\preccurlyeq}\alpha_1\cup \cdots\cup {\uparrow}_{\preccurlyeq}\alpha_{i}\right)\cap U_X(x)$,
\end{itemize}
for all positive integers $i>1$.

Then Lemma~\ref{lemma-2.1} implies that the infinite family of non-empty open subsets $\mathscr{U}=\left\{\left\{\gamma_i\right\}\colon i\in\mathbb{N}\right\}$ is locally finite, which contradicts the feeble compactness of ${\uparrow}_{\preccurlyeq}\beta$. The obtained contradiction implies the statement of induction which completes the proof of the statement that the space $\left(\mathscr{I}_\lambda^n,\tau\right)$ is H-closed.

$(iv)\Rightarrow(v)$ By Kat\v{e}tov's Theorem every continuous image of an H-closed topological space into a Hausdorff space is H-closed (see \cite[3.15.5~(b)]{Engelking-1989} or \cite{Katetov-1940}). Hence the image $f(\mathscr{I}_\lambda^n)$ is H-closed for every continuous map $f\colon \left(\mathscr{I}_\lambda^n,\tau\right)\to\mathbb{N}_{\mathfrak{d}}$, which implies that $f(\mathscr{I}_\lambda^n)$ is compact (see \cite[3.15.5~(a)]{Engelking-1989}).

$(v)\Rightarrow(ii)$ Suppose to the contrary that there exists a Hausdorff shift-continuous $\mathbb{N}_{\mathfrak{d}}$-compact topology $\tau$ on  $\mathscr{I}_\lambda^n$ which is not feebly compact. Then there exists an infinite locally finite family $\mathscr{U}=\{U_i\}$ of open non-empty subsets of $\left(\mathscr{I}_\lambda^n,\tau\right)$. Without loss of generality we may assume that the family $\mathscr{U}=\{U_i\}$ is countable., i.e., $\mathscr{U}=\{U_i\colon i\in\mathbb{N}\}$. Then the definition of the semigroup $\mathscr{I}_\lambda^n$ and Lemma~\ref{lemma-2.2} imply that for every $U_i\in\mathscr{U}$ there exists $\alpha_i\in U_i$ such that ${\uparrow}_{\preccurlyeq}\alpha_i\cap U_i=\{\alpha_i\}$ and hence $\mathscr{U}^*=\{\{\alpha_i\}\colon i\in\mathbb{N}\}$ is a family of isolated points of $\left(\mathscr{I}_\lambda^n,\tau\right)$. Since the family $\mathscr{U}$ is locally finite, without loss of generality we may assume that $\alpha_i\neq\alpha_j$ for distinct $i,j\in\mathbb{N}$. We claim that the family $\mathscr{U}^*$ is locally finite. Indeed, if we assume the contrary then there exists $\alpha\in \mathscr{I}_\lambda^n$ such that every open neighbourhood of $\alpha$ contains infinitely many elements of the family $\mathscr{U}^*$. This implies that the family $\mathscr{U}$ is not locally finite, a contradiction. Since  $\left(\mathscr{I}_\lambda^n,\tau\right)$ is a $T_1$-space and the family $\mathscr{U}^*$ is locally finite, we have that $\bigcup\mathscr{U}^*$ is a closed subset in $\left(\mathscr{I}_\lambda^n,\tau\right)$ and hence the map $f\colon \left(\mathscr{I}_\lambda^n,\tau\right)\to \mathbb{N}_{\mathfrak{d}}$ defined by the formula
\begin{equation*}
  f(\beta)=
\left\{
  \begin{array}{cl}
    1, & \hbox{if~} \beta\in\mathscr{I}_\lambda^n\setminus\bigcup\mathscr{U}^*;\\
    i+1, & \hbox{if~} \beta=\alpha_i \hbox{~for some~} i\in\mathbb{N},
  \end{array}
\right.
\end{equation*}
is continuous. This contradicts that the space $\left(\mathscr{I}_\lambda^n,\tau\right)$ is $\mathbb{N}_{\mathfrak{d}}$-compact.

The proofs of implications $(iv)\Rightarrow(vi)$ and $(vi)\Rightarrow(ii)$ are same as the proofs of $(iv)\Rightarrow(v)$ and $(v)\Rightarrow(ii)$, respectively.

Implication $(ii)\Rightarrow(vii)$ follows from Proposition 2 and Theorem 3 of \cite{Hajek-Todd-1975}.

$(vii)\Rightarrow(ii)$ Suppose to the contrary that there exists a Hausdorff shift-continuous infra H-closed topology $\tau$ on $\mathscr{I}_\lambda^n$ which is not feebly compact. Then similarly as in the proof of implication $(v)\Rightarrow(ii)$ we choice a locally finite family $\mathscr{U}^*=\{\{\alpha_i\}\colon i\in\mathbb{N}\}$ of isolated points of $\left(\mathscr{I}_\lambda^n,\tau\right)$. Then the map $f\colon \left(\mathscr{I}_\lambda^n,\tau\right)\to \mathbb{R}$ defined by the formula
\begin{equation*}
  f(\beta)=
\left\{
  \begin{array}{cl}
    1, & \hbox{if~} \beta\in\mathscr{I}_\lambda^n\setminus\bigcup\mathscr{U}^*;\\
    \dfrac{1}{i+1}, & \hbox{if~} \beta=\alpha_i \hbox{~for some~} i\in\mathbb{N},
  \end{array}
\right.
\end{equation*}
is continuous. This contradicts that the space $\left(\mathscr{I}_\lambda^n,\tau\right)$ is infra H-closed.
\end{proof}

\begin{remark}\label{remark-2.9a}
By Theorem~5 from \cite{Hajek-Todd-1975} conditions $(ii)$ and $(vii)$ of Theorem~\ref{theorem-2.9} are equivalent for any Tychonoff space $X$.

It is not, however, the case that feebly compact and infra H-closed are equivalent in general. In \cite{Herrlich-1965} Herrlich, beginning with a $T_1$-space $Y$, constructs a regular space $X$ such that the only continuous functions from $X$ into $Y$ are constant. His construction involves the cardinality of $Y$, but only as the cardinality of collections of open sets whose intersections are singletons. Thus only the most trivial modifications are needed in his argument to produce a regular Hausdorff infra H-closed space. It is also easily shown that the space constructed in this manner is not feebly compact.

Any regular lightly compact space must be a Baire space \cite[Lemma~3]{McCoy-1973}, and thus it is of interest to note that the space constructed in \cite{Herrlich-1965} can be shown to be the countable union of nowhere dense subsets using essentially the same argument as can be used to show it is not feebly compact.
\end{remark}

Later we need the following technical lemma.

\begin{lemma}\label{lemma-2.10}
Let $n$ be an arbitrary positive integer and $\lambda$ be an arbitrary infinite cardinal. Let $\tau$ be a feebly compact shift-continuous $T_1$-topology on the semigroup $\mathscr{I}_\lambda^n$. Then for every $\alpha\in\mathscr{I}_\lambda^n$ and any open neighbourhood $U(\alpha)$ of $\alpha$ in $\left(\mathscr{I}_\lambda^n,\tau\right)$ there exist finitely many $\alpha_1,\ldots,\alpha_k\in{\uparrow}_{\preccurlyeq}\alpha\setminus\{\alpha\}$ such that
\begin{equation*}
  \mathscr{I}_\lambda^n\setminus\mathscr{I}_\lambda^{n-1}\cap {\uparrow}_{\preccurlyeq}\alpha \subseteq U(\alpha)\cup{\uparrow}_{\preccurlyeq}\alpha_1\cup \cdots\cup {\uparrow}_{\preccurlyeq}\alpha_k.
\end{equation*}
\end{lemma}

\begin{proof} 
Suppose to the contrary that there exists a feebly compact shift-continuous $T_1$-topology on the semigroup $\mathscr{I}_\lambda^n$ which satisfies the following property: some element $\alpha$ of the semigroup $\mathscr{I}_\lambda^n$ has an open neighbourhood $U(\alpha)$ of $\alpha$ in $\left(\mathscr{I}_\lambda^n,\tau\right)$ such that
\begin{equation*}
  \mathscr{I}_\lambda^n\setminus\mathscr{I}_\lambda^{n-1}\cap {\uparrow}_{\preccurlyeq}\alpha\nsubseteq U(\alpha)\cup{\uparrow}_{\preccurlyeq}\alpha_1\cup \cdots\cup {\uparrow}_{\preccurlyeq}\alpha_k,
\end{equation*}
for any finitely many $\alpha_1,\ldots,\alpha_k\in{\uparrow}_{\preccurlyeq}\alpha\setminus\{\alpha\}$. We observe that Lemma~\ref{lemma-2.2} implies that without loss of generality we may assume that $U(\alpha)\subseteq{\uparrow}_{\preccurlyeq}\alpha$.

Fix such an element $\alpha$ of $\mathscr{I}_\lambda^n$ and its open neighbourhood $U(\alpha)$ with the above determined property. Then our assumption implies that there exists $\alpha_1\in{\uparrow}_{\preccurlyeq}\alpha\setminus U(\alpha)$ such that the set
\begin{equation*}
  \mathscr{I}_\lambda^n\setminus\mathscr{I}_\lambda^{n-1}\cap{\uparrow}_{\preccurlyeq}\alpha_1\setminus U(\alpha)
\end{equation*}
is infinite and fix an arbitrary $\gamma_1\in\mathscr{I}_\lambda^n\setminus\mathscr{I}_\lambda^{n-1}\cap{\uparrow}_{\preccurlyeq}\alpha_1\setminus U(\alpha)$. Next, by induction using our assumption we construct sequences $\left\{\alpha_i\right\}_{i\in\mathbb{N}}$ and $\left\{\gamma_i\right\}_{i\in\mathbb{N}}$ of the distinct points of the set ${\uparrow}_{\preccurlyeq}\alpha$ such that the following conditions hold:
\begin{itemize}
  \item[$(a)$] $\alpha_{i+1}\in{\uparrow}_{\preccurlyeq}\alpha\setminus \left(U(\alpha)\cup{\uparrow}_{\preccurlyeq}\alpha_1\cup\cdots\cup {\uparrow}_{\preccurlyeq}\alpha_i\right)$ and the set
\begin{equation*}
  \mathscr{I}_\lambda^n\setminus\mathscr{I}_\lambda^{n-1}\cap{\uparrow}_{\preccurlyeq}\alpha \cap{\uparrow}_{\preccurlyeq}\alpha_{1+1}\setminus \left(U(\alpha)\cup{\uparrow}_{\preccurlyeq}\alpha_1\cup\cdots\cup {\uparrow}_{\preccurlyeq}\alpha_i\right)
\end{equation*}
is infinite;
  \item[$(b)$] $\gamma_{i+1}\in\mathscr{I}_\lambda^n\setminus\mathscr{I}_\lambda^{n-1}\cap {\uparrow}_{\preccurlyeq}\alpha \cap{\uparrow}_{\preccurlyeq}\alpha_1\setminus \left(U(\alpha)\cup{\uparrow}_{\preccurlyeq}\alpha_1\cup\cdots\cup {\uparrow}_{\preccurlyeq}\alpha_i\right)$,
\end{itemize}
for all positive integers $i$.

By Lemma~\ref{lemma-2.2} and Theorem~14 of \cite{Bagley-Connell-McKnight-Jr-1958} the space ${\uparrow}_{\preccurlyeq}\alpha$ is feebly compact. Then Lemma~\ref{lemma-2.1} implies that the infinite family of non-empty open subsets $\mathscr{U}=\left\{\left\{\gamma_i\right\}\colon i\in\mathbb{N}\right\}$ is locally finite, which contradicts the feeble compactness of ${\uparrow}_{\preccurlyeq}\alpha$. The obtained contradiction implies the statement of the lemma.
\end{proof}

\begin{theorem}\label{theorem-2.11}
Let $n$ be an arbitrary positive integer and $\lambda$ be an arbitrary infinite cardinal. Then every shift-continuous semiregular feebly compact $T_1$-topology $\tau$ on $\mathscr{I}_\lambda^n$ is compact.
\end{theorem}

\begin{proof} 
We shall prove the statement of the theorem by induction. First we observe that for every element $\alpha$ of a semiregular feebly compact $T_1$-semitopological semigroup $\left(\mathscr{I}_\lambda^n,\tau\right)$ with $\operatorname{rank}\alpha=n-1, n$ the set ${\uparrow}_{\preccurlyeq}\alpha$ is compact. Indeed, by Lemma~\ref{lemma-2.2} for every $\beta\in \mathscr{I}_\lambda^n$ the set ${\uparrow}_{\preccurlyeq}\beta$ is open-and-closed in $\left(\mathscr{I}_\lambda^n,\tau\right)$, and hence we have that $\mathscr{I}_\lambda^n\setminus\mathscr{I}_\lambda^{n-1}$ is an open discrete subspace of $\left(\mathscr{I}_\lambda^n,\tau\right)$ and using Theorem~14 of \cite{Bagley-Connell-McKnight-Jr-1958} we obtain that ${\uparrow}_{\preccurlyeq}\alpha$ is feebly compact, which implies that the space ${\uparrow}_{\preccurlyeq}\alpha$ is compact.

Next we shall prove a more stronger step of induction: \emph{if for every element $\alpha$ of a semiregular feebly compact $T_1$-semitopological semigroup $\left(\mathscr{I}_\lambda^n,\tau\right)$ with $\operatorname{rank}\alpha>l\leqslant n$ the set ${\uparrow}_{\preccurlyeq}\alpha$ is compact, then ${\uparrow}_{\preccurlyeq}\beta$ is compact for every $\beta\in\mathscr{I}_\lambda^n$ with $\operatorname{rank}\alpha=l$.}

Suppose to the contrary that there exists a semiregular feebly compact $T_1$-semitopological semigroup $\left(\mathscr{I}_\lambda^n,\tau\right)$ such that for some positive integer $l\leqslant n$ the set ${\uparrow}_{\preccurlyeq}\alpha$ is compact for every $\alpha\in\mathscr{I}_\lambda^n$ with $\operatorname{rank}\alpha>l$, but there exists $\beta\in\mathscr{I}_\lambda^n$ with $\operatorname{rank}\beta=l$ such that the set  ${\uparrow}_{\preccurlyeq}\beta$ is not compact.

First we observe that our assumption that the set ${\uparrow}_{\preccurlyeq}\alpha$ is compact and  Corolla\-ry~3.1.14 of \cite{Engelking-1989} imply that  the following family
\begin{equation*}
    \mathscr{B}_{\operatorname{\textsf{c}}}(\alpha)=\left\{U_\alpha(\alpha_1,\ldots,\alpha_k)=
    {\uparrow}_{\preccurlyeq}\alpha\setminus({\uparrow}_{\preccurlyeq}\alpha_1\cup\cdots\cup
    {\uparrow}_{\preccurlyeq}\alpha_k)\colon \alpha_i\in{\uparrow}_{\preccurlyeq}\alpha\setminus
    \{\alpha\},  i=1,\ldots, k\right\}
\end{equation*}
is a base of topology at the point $\alpha$ of $\left(\mathscr{I}_\lambda^n,\tau\right)$ for every $\alpha\in\mathscr{I}_\lambda^n$ with $\operatorname{rank}\alpha>l$.

Then the Alexander Subbase Theorem (see \cite[Theorem~1]{Alexander-1939} or \cite[p.~221, 3.12.2(a)]{Engelking-1989}) and Lemma~\ref{lemma-2.2} imply that there exists a base $\mathscr{B}$ of the topology $\tau$ on $\mathscr{I}_\lambda^n$ with the following properties:
\begin{itemize}
  \item[$(i)$] $\mathscr{B}=\bigcup\left\{\mathscr{B}(\gamma)\colon\gamma\in\mathscr{I}_\lambda^n \right\}$ and for every $\gamma\in\mathscr{I}_\lambda^n$ the family $\mathscr{B}(\gamma)$ is a base at the point $\gamma$;
  \item[$(ii)$] $U(\gamma)\subseteq {\uparrow}_{\preccurlyeq} \gamma$ for any $U(\gamma)\in\mathscr{B}(\gamma)$;
  \item[$(iii)$] $\mathscr{B}(\gamma)=\mathscr{B}_{\operatorname{\textsf{c}}}(\gamma)$ for every $\gamma\in\mathscr{I}_\lambda^n$ with $\operatorname{rank}\gamma>l$;
  \item[$(iv)$] there exists a cover $\mathscr{U}$ of the set  ${\uparrow}_{\preccurlyeq}\beta$ by members of the base $\mathscr{B}$ which has not a finite subcover.
\end{itemize}

We claim that the subspace ${\uparrow}_{\preccurlyeq}\beta$ of $\left(\mathscr{I}_\lambda^n,\tau\right)$ contains an infinite closed discrete subspace $X$. Indeed, let $\mathscr{U}_0$ be a subfamily of $\mathscr{U}$ such that
\begin{equation*}
\{\beta\}\cup \left({\uparrow}_{\preccurlyeq}\beta\cap\mathscr{I}_\lambda^n\setminus\mathscr{I}_\lambda^{n-1}\right)\subseteq \bigcup\mathscr{U}_0.
\end{equation*}
Since the set  ${\uparrow}_{\preccurlyeq}\beta$ is not compact and ${\uparrow}_{\preccurlyeq} \gamma$ is compact for any $\gamma\in {\uparrow}_{\preccurlyeq}\beta\setminus\{\beta\}$, without loss of generality we may assume that there exists $k>\operatorname{rank}\beta$ such that the following conditions hold:
\begin{itemize}
  \item[$(a)$] there exist infinitely many elements $\zeta\in{\uparrow}_{\preccurlyeq}\beta$ with $\operatorname{rank}\zeta=k$ such that $\zeta\notin \bigcup\mathscr{U}_0$;
  \item[$(b)$] $\varsigma\in\bigcup\mathscr{U}_0$ for all $\varsigma\in{\uparrow}_{\preccurlyeq}\beta$ with $\operatorname{rank}\varsigma<k$.
\end{itemize}
It is obvious that the set
\begin{equation*}
X={\uparrow}_{\preccurlyeq}\beta\setminus\left(\bigcup\mathscr{U}_0\setminus \bigcup\left\{{\uparrow}_{\preccurlyeq}\varsigma\colon \operatorname{rank}\varsigma>k\right\}\right)
\end{equation*}
is requested.

Fix an arbitrary regular open neighbourhood $U(\beta)$ of the point $\beta$ in $\left(\mathscr{I}_\lambda^n,\tau\right)$ such that $U(\beta)\cap X=\varnothing$. By Lemma~\ref{lemma-2.10} there exist finitely many $\beta_1,\ldots,\beta_s\in{\uparrow}_{\preccurlyeq}\beta$ such that
\begin{equation*}
  \mathscr{I}_\lambda^n\setminus\mathscr{I}_\lambda^{n-1}\cap {\uparrow}_{\preccurlyeq}\beta \subseteq U(\beta)\cup{\uparrow}_{\preccurlyeq}\beta_1\cup \cdots\cup {\uparrow}_{\preccurlyeq}\beta_s.
\end{equation*}
It is obvious that the set $X\setminus({\uparrow}_{\preccurlyeq}\beta_1\cup \cdots\cup {\uparrow}_{\preccurlyeq}\beta_s)$ is infinite. For every $\delta\in X$ the set ${\uparrow}_{\preccurlyeq}\delta$ is compact and open, and moreover by Lemma~\ref{lemma-2.1} the set ${\uparrow}_{\preccurlyeq}\delta\setminus\left({\uparrow}_{\preccurlyeq}\beta_1\cup \cdots\cup {\uparrow}_{\preccurlyeq}\beta_s\right)$ contains infinitely many points of the neighbourhood $U(\beta)$. This implies that $\operatorname{int}_{\mathscr{I}_\lambda^n}\left(\operatorname{cl}_{\mathscr{I}_\lambda^n}(U(\beta))\right)\cap X\neq\varnothing$, which contradicts the assumption that $U(0)\cap X=\varnothing$. The obtained contradiction implies that the subspace ${\uparrow}_{\preccurlyeq}\beta$ of $\left(\mathscr{I}_\lambda^n,\tau\right)$ is compact, which completes the proof of the theorem.
\end{proof}




\medskip
\paragraph*{Acknowledgements}
The author acknowledges Taras Banakh and the referee for useful important comments and suggestions.

\end{document}